\documentclass[12pt]{amsart}

\usepackage{amsmath}
\usepackage{amsfonts}
\usepackage{amssymb}
\usepackage{amsthm}
\usepackage{hyperref}
\usepackage{enumerate}
\usepackage{cleveref}
\usepackage{mathrsfs}

\newtheorem{theorem}{Theorem}[section]
\newtheorem{lemma}[theorem]{Lemma}
\newtheorem{proposition}[theorem]{Proposition}
\newtheorem{corollary}[theorem]{Corollary}
\newtheorem{definition}[theorem]{Definition}
\newtheorem{remark}[theorem]{Remark}

\Crefname{conjecture}{Conjecture}{Conjectures}

\theoremstyle{plain}

\theoremstyle{plain}

\newcommand{\N}{\mathbb{N}}

\newcommand{\Z}{\mathbb{Z}}
\newcommand{\Q}{\mathbb{Q}}
\newcommand{\R}{\mathbb{R}}
\newcommand{\C}{\mathbb{C}}

\newcommand{\bbH}{\mathbb{H}}
\newcommand{\calO}{\mathcal{O}}
\newcommand{\mymod}{\operatorname{mod}}

\newcommand{\eps}{\varepsilon}

\newcommand{\SL}{\operatorname{SL}}

\newcommand{\calH}{\mathscr{H}}

\newcommand{\M}{\mathcal{M}}

\renewcommand{\Im}{\operatorname{Im}}
\newcommand{\calR}{\mathscr{R}}
\newcommand{\SLZ}{\SL_2(\Z)}
\newcommand{\abcd}{\left(\begin{smallmatrix} a & b \\ c & d \end{smallmatrix}\right)}

\newcommand{\calS}{\mathcal{S}}
\newcommand{\pihol}{\pi_{hol}}
\newcommand{\tr}{\operatorname{trace}}

\numberwithin{equation}{section}

\author{Michael H. Mertens}
\title{Eichler-Selberg Type Identities for Mixed Mock Modular Forms}
\address{Mathematisches Institut der Universit\"at zu K\"oln \\ 
         Weyertal 86-90 \\
         D-50931 K\"oln, Germany}
\email{mmertens@math.uni-koeln.de}
\urladdr{http://www.mi.uni-koeln.de/~mmertens}
\thanks{The author's research is supported by the DFG-Graduiertenkolleg 1269 "Global Structures in Geometry and Analysis".\\
This paper is part of the author's Ph.D. thesis \cite{MertPhD}, written under the supervision of Prof. Dr. K.~Bringmann at the Universit\"at zu K\"oln.}

\begin{document}
\maketitle
\centerline{\bf Abstract}

\noindent
Using holomorphic projection, we work out a parametrization for all relations of products (resp. Rankin-Cohen brackets) of weight $\tfrac 32$ mock modular forms with holomorphic shadow and weight $\tfrac 12$ modular forms in the spirit of the Kronecker-Hurwitz class number relations. In particular we obtain new proofs for several class number relations among which some are classical, others are relatively new. We also obtain similar results for the mock theta functions.

\section{Introduction}
Throughout the last 90 years, a great deal of effort of mathematical research has been spent on the \emph{mock theta functions}, which were introduced by S. Ramanujan in his by now world famous deathbed letter to G. H. Hardy (cf. \cite[pp. 220-224]{BR95}). It was only some years ago when it was finally revealed in work of S. Zwegers \cite{ZwegersDiss}, J. H. Bruinier and J. Funke \cite{BF04}, K. Bringmann and K. Ono \cite{BO06, BO10}, and many others what the actual nature of these mock theta function or more generally \emph{mock modular forms} is: they are holomorphic parts of so called \emph{harmonic weak Maa{\ss} forms} (cf. \Cref{sec:Mock}). Mock modular forms have since then had vast applications in partition theory \cite{BL09, BO10}, theory of Lie superalgebras \cite{BF13}, and mathematical physics, e.g. in quantum black holes \cite{Zag12}, just to name a few.

A very famous example of a mock modular form of weight $\tfrac 32$ is the Hurwitz class number generating function
\[\calH(\tau):=\sum\limits_{n=0}^\infty H(n)q^n,\quad \Im(\tau)>0,\quad q:=e^{2\pi i\tau},\]
where $H(0):=-\tfrac{1}{12}$ and for $n\in\N$, $H(n)$ denotes the Hurwitz class number, i.e. the class number of binary integral quadratic forms of discriminant $-n$, where the class containing $a(x^2+y^2)$ (resp. $a(x^2+xy+y^2)$) is counted with multiplicity $\tfrac 12$ (resp. $\tfrac 13$). For convenience we set $H(n)=0$ for $n\notin\N_0$. 

This function was historically the first fully understood example of a mock modular form, without the terminology having been introduced at the time \cite{HZ76}. In \cite{Mert13}, the author used the above result as well as the properties of so called \emph{Appell-Lerch sums} first extensively studied in this context by S. Zwegers in \cite{ZwegersDiss, Zwegers10} to give a mock modular proof of an infinite family of class number relations for odd numbers $n$: 
\begin{align}
\label{eq:Eichler}
\sum\limits_{s\in\Z} H(n-s^2) +\lambda_1(n)&=\frac{1}{3}\sigma_1(n),\\
\label{eq:Cohen}
\sum\limits_{s\in\Z} \left(4s^2-n\right)  H\left(n-s^2\right)+\lambda_3\left(n\right)&=0
\end{align}
where
\begin{gather}\label{eq:lambda}
\lambda_k(n):=\frac 12\sum\limits_{d\mid n} \min\left(d,\frac{n}{d}\right)^k
\end{gather}
and $\sigma_k(n)$ is the usual $k$th power divisor sum. The relation in \eqref{eq:Eichler} was first found by M. Eichler in \cite{Eichler55}, the one in \eqref{eq:Cohen} and infinitely many more had been conjectured by H. Cohen in \cite{Coh75}. Because of their resemblance to the famous class number relation due to Kronecker \cite{Kro60} and Hurwitz \cite{Hur84,Hur85}
\[\sum\limits_{s\in\Z} H(4n-s^2)-2\lambda_1(n)=2\sigma_1(n)\]
and the ones obtained from the Eichler-Selberg trace formula, we refer to these as \emph{Eichler-Selberg type relations}.

In this article, we establish that the Fourier coefficients of all mock modular forms of weight $\tfrac 32$ with holomorphic shadow fulfill Kronecker-Hurwitz type relations (cf. \Cref{theo:main32}). In \Cref{theo:main12} we prove a similar result for the mock theta functions.

The main ingredients for this are the Theorem of Serre and Stark \cite[Theorem A]{SS77} which states that every modular form of weight $\tfrac 12$ is a linear combination of theta series, and holomorphic projection (see \Cref{sec:HolProj}). 

The paper is organized as follows: In \Cref{sec:Mock} we give a brief account of some important facts about harmonic Maa{\ss} forms and mock modular forms and in \Cref{sec:HolProj} we introduce holomorphic projection and work out its action on Rankin-Cohen brackets of mock modular forms and modular forms of arbitrary (positive) weight in two different ways. The first one follows notes of D. Zagier \cite{ZagNotes} which the author was kindly allowed to use, the second one is a little bit more subtle but involves less computation. The special results for the cases that we are interested in are worked out in \Cref{sec:weight,sec:MockTheta}. As applications to our main result, we reprove the class number relations from \cite{Coh75,Mert13}, from the Eichler-Selberg trace formula, and generalizations of the ones from \cite{BK13} in \Cref{sec:Ex} in a more natural way than the method in \cite{BK13,Mert13}. 
\section{Some Preliminaries}
Several proofs in the present paper consist of calculations involving the Gamma function and hypergeometric series. Let us therefore recall some standard notation and some useful identities.

For the Gamma function we have the following well-known \emph{duplication formula} due to A. M. Legendre (cf. \cite{FL05}, Satz 5.4.), which holds for all $s\in\C\setminus\left(-\tfrac 12\N_0\right)$,
\begin{gather}\label{eq:Legendre}
\Gamma(2s)=\frac{1}{\sqrt{\pi}}2^{2s-1}\Gamma(s)\Gamma\left(s+\frac 12\right).
\end{gather} 
Apart from this, there is also the following formula which is sometimes called the \emph{second functional equation} of the Gamma function or the \emph{reflection formula} and goes back to L. Euler (cf. \cite{FL05}, Satz 5.1., \cite{FB06}, p. 204). For all $s\in\C\setminus\Z$ we have that
\begin{gather}\label{eq:Gamma}
\Gamma(s)\Gamma(1-s)=\frac{\pi}{\sin(\pi s)}.
\end{gather}
Formally, we define the \emph{generalized hypergeometric series} by
\begin{gather}\label{eq:qFp}
{}_pF_q\left(\begin{matrix} a_1,\dots ,a_p \\ b_1,\dots ,b_q\end{matrix} ; x\right):=\sum\limits_{n=0}^\infty \frac{(a_1)_n\cdots (a_p)_n}{(b_1)_n\cdots (b_q)_n}\cdot\frac{x^n}{n!},\qquad b_j\notin-\N_0
\end{gather}
where
\[(a)_n:=\prod\limits_{j=0}^{n-1} (a+j)=\frac{\Gamma(a+n)}{\Gamma(a)}\]
denotes the \emph{Pochhammer symbol}. It is worth pointing out that every (convergent) sum $\sum\limits_{n=0}^\infty c_n$ ($c_n\in\C$) where the quotient $\tfrac{c_{n+1}}{c_n}$ is a rational function in $n$ can be written as a multiple of a hypergeometric series (cf. \cite{AAR}, pp. 61 f).

A useful device to evaluate special hypergeometric series is the \emph{Pfaff-Saalsch\"utz identity} (cf. \cite{AAR}, Theorem 2.2.6).
\begin{gather}\label{eq:Pfaff}
{}_3F_2\left(\begin{matrix} -n,a,b \\ c,1+a+b-c-n \end{matrix};1\right)=\frac{(c-a)_n(c-b)_n}{(c)_n(c-a-b)_n}.
\end{gather}
\section{Harmonic Maa{\ss} Forms and Mock Modular Forms}\label{sec:Mock}
In this section, we briefly recall some basic facts about harmonic Maa{\ss} forms. More detailed information may be found in \cite{Zag12, Ono}.

Let $\bbH$ denote the complex upper half-plane and $f:\bbH\rightarrow\C$ be a smooth function, $\gamma=\abcd\in\SL_2(\Z)$, and $k\in\tfrac 12\Z$. We define the following three operators.
\begin{enumerate}
\item The weight $k$ \emph{slash operator}: 
\[(f|_k\gamma)(\tau):=\begin{cases} (c\tau+d)^{-k}f\left(\frac{a\tau+b}{c\tau+d}\right), &\text{if }k\in\Z \\
                                      \left(\frac{c}{d}\right)\eps_d\left(\sqrt{c\tau+d}\right)^{-2k}f\left(\frac{a\tau+b}{c\tau+d}\right), &\text{if }k\in\tfrac 12+\Z,
                                      \end{cases}\]
For $k\notin\Z$ we assume $\gamma\in\Gamma_0(4)$ and we set $\left(\tfrac mn\right)$ to be the extended Legendre symbol in the sense of \cite{Shi}, $\sqrt{\tau}$ the principal branch of the holomorphic square root (i.e. $-\tfrac\pi 2<\arg(\sqrt{\tau})\leq\tfrac \pi 2$), and
\[\eps_d:=\begin{cases} 1 & \text{, if }d\equiv 1\quad (\mymod 4)\\ i & \text{, if }d\equiv 3\quad (\mymod 4).\end{cases}\]
\item The weight $k$ \emph{hyperbolic Laplacian} (where from now on $\tau=x+iy$, $x\in\R,\,y>0$)
\[\Delta_k:=-y^2\left(\frac{\partial^2}{\partial x^2}+\frac{\partial^2}{\partial y^2}\right)+iky\left(\frac{\partial}{\partial x}+i\frac{\partial}{\partial y}\right),\]                                      
\item The weight $k$ \emph{$\xi$-operator}
\[\xi_{k}:=2iy^k\overline{\frac{\partial}{\partial\overline{\tau}}}.\]
\end{enumerate}
\begin{definition}
A smooth function $f\colon\bbH\rightarrow \C$ is called a \emph{harmonic weak Maa{\ss} form} of weight $k\in\tfrac 12\Z$, level $N\in\N$, and character $\chi$ modulo $N$ (with $4\mid N$ if $k\notin\Z$) if it fulfills the following properties:
\begin{enumerate}
\item $f$ is $\Gamma_0(N)$-equivariant, i.e. $(f|_k\gamma)(\tau)=\chi(d)f(\tau)$ for all $\gamma=\abcd\in\Gamma_0(N)$,
\item $f$ lies in the kernel of the hyperbolic Laplacian, i.e. $\Delta_kf\equiv 0$,
\item $f$ grows at most linearly exponentially approaching the cusps of $\Gamma_0(N)$.
\end{enumerate}
The vector space of harmonic weak Maa{\ss} forms of weight $k$, level $N$, and character $\chi$ is denoted by $\mathcal{H}_k(N,\chi)$.
\end{definition}
Exploiting the fact that a harmonic weak Maa{\ss} form is anihilated by $\Delta_k$, one finds that every harmonic Maa{\ss} form splits in the following way.
\begin{lemma}[\cite{Ono}, Lemma 7.2]\label{prop:split}
Let $f$ be a harmonic weak Maa{\ss} form of weight $k\neq 1$. Then $f$ has a canonical splitting into
\begin{gather}\label{eq:split}
f(\tau)=f^+(\tau)+\frac{(4\pi y)^{1-k}}{k-1}\overline{c_f^-(0)}+f^-(\tau),
\end{gather}
where for some $M\in\Z$ we have the Fourier expansions
\[f^+(\tau)=\sum\limits_{n=M}^\infty c_f^+(n)q^n\]
and 
\[f^-(\tau)=\sum\limits_{n=1}^\infty \overline{c_f^-(n)}n^{k-1}\Gamma(1-k;4\pi ny)q^{-n}.\]
As usually we set $q:=e^{2\pi i\tau}$ and 
\[\Gamma(\alpha;x):=\int\limits_x^\infty t^{\alpha-1}e^{-t} dt\]
denotes the incomplete Gamma function. 
\end{lemma}
The functions $f^+$ (resp. $\tfrac{(4\pi y)^{1-k}}{k-1}\overline{c_f^-(0)}+f^-(\tau)$) in \Cref{prop:split} are referred to as the \emph{holomorphic} (resp. \emph{non-holomorphic}) \emph{part} of the harmonic Maa{\ss} form $f$. We define a \emph{mock modular form} of weight $k$ to be the holomorphic part of a harmonic Maa{\ss} form of the same weight. In weight $k=\tfrac 12$, we shall call $f^+$ a \emph{mock theta function} if its shadow (see below) is a unary theta function of weight $\tfrac 32$, i.e. of the form $\theta_{\chi,s}=\sum\limits_{n\in\Z} \chi(n)nq^{sn^2}$ for $s\in\N$ and $\chi$ an odd character. 

Let us also recall the following result.
\begin{proposition}[\cite{BF04}, Proposition 3.2]
For $k\neq 1$, the mapping 
\[\xi_k:\mathcal{H}_k(N,\chi)\rightarrow M_{2-k}^!(N,\overline{\chi}),\:f\mapsto\xi_kf\]
is well-defined and surjective with kernel $M_k^!(N,\chi)$, where $M_\ell^!$ denotes the space of weakly holomorphic modular forms of weight $\ell$. Moreover, for $f\in\mathcal{H}_k(N,\chi)$, we have that
\[(\xi_kf)(\tau)=-(4\pi)^{1-k}\sum\limits_{n=0}^\infty c_f^-(n)q^n.\]
\end{proposition}
We call the function $(\xi_k f)$ the \emph{shadow} of $f$ (or of the mock modular form $f^+$) and denote the preimages of $M_{2-k}(N,\overline{\chi})$, the space of holomorphic modular forms, (resp. $S_{2-k}(N,\overline{\chi})$, the space of cusp forms) by $\M_k(N,\chi)$ (resp. $\calS_k(N,\chi)$). By $\M_k^{mock}(N,\chi)$ (resp. $\calS^{mock}_k(N,\chi)$) we denote the spaces of the respective holomorphic parts.

To conclude this section, we define the object of main interest in this paper.
\begin{definition}\label{def:mixed}
Let $f$ be a mock modular form of weight $k$ and $g$ be a holomorphic modular form of weight $\ell$. 
\begin{enumerate}[(i)]
\item The product $f\cdot g$ is called a \emph{mixed mock modular form} of weight $(k,\ell)$.
\item More generally, the $\nu$th Rankin-Cohen bracket 
\[[f,g]_\nu:=\sum\limits_{\mu=0}^\nu (-1)^\mu {{k+\nu-1} \choose \nu-\mu} {{\ell+\nu-1} \choose \mu} D^\mu fD^{\nu-\mu}g,\]
with $D:=D_\tau:=\tfrac{1}{2\pi i}\tfrac{d}{d\tau}$ of $f$ and $g$ is called a mixed mock modular form of weight $(k,\ell)$ of \emph{degree} $\nu$.
\end{enumerate}
\end{definition}
This terminology is motivated by the fact that each Rankin-Cohen bracket defines a more or less natural product on modular forms. Recall that for (real-analytic) modular forms $f$ and $g$ of weights $k$ and $\ell$, we have that $[f,g]_0=f\cdot g$ and that $[f,g]_\nu$ is modular of weight $k+\ell+2\nu$ (cf. \cite{Coh75}, Theorem 7.1), hence e.g. the completion of a mixed mock modular form of weight $(k,\ell)$ and degree $\nu$ is a real-analytic modular form of weight $k+\ell+2\nu$. 

\section{Holomorphic Projection}\label{sec:HolProj}
Here, we investigate the properties of the holomorphic projection operator introduced by J. Sturm \cite{St80} and further developed in \cite{GrZ86}. Using an idea of S. Zwegers, D. Zagier has worked out the action of this operator on mixed mock modular forms, and also higher degree mixed mock modular forms \cite{ZagNotes}. Recently in \cite{IRR13}, \"O. Imamo\u{g}lu, M. Raum, and O. Richter extended this in degree $0$, i.e. usual mixed mock modular forms, to the case of vector-valued forms. Special cases of our main results also obtained by holomorphic projection are also contained in \cite{AA14,ARZ}

Since Zagier's results have not been published so far in full generality but seem to be useful in a broader context, we give an account of them here.
\begin{definition}\label{def:pihol}
Let $f:\bbH\rightarrow\C$ be a continuous function transforming like a modular form of weight $k\geq 2$ on some $\Gamma_0(N)$ with Fourier expansion
\[f(\tau)=\sum\limits_{n\in\Z} a_f(n,y)q^n.\]
For a cusp $\kappa_j$, $j=1,...,M$ and $\kappa_1:=i\infty$, of $\Gamma_0(N)$ fix $\gamma_j\in\SL_2(\Z)$ with $\gamma_j\kappa_j=i\infty$. Assume that for some $\delta,\eps>0$, and $k\in\tfrac 12\Z$, ($k\geq 2$) we have
\begin{enumerate}
\item $f(\gamma_j^{-1}w)\left(\frac d{dw}\tau\right)^\frac k2=c_0^{(j)}+O(\Im(w)^{-\delta})$
for all $j=1,...,M$ and $w=\gamma_j\tau$,
\item $a_f(n,y)=\calO(y^{1-k+\eps})$ as $y\rightarrow 0$ for all $n>0$.
\end{enumerate}
Then we define the \emph{holomorphic projection} of $f$ by
\[
(\pihol f)(\tau):=(\pihol^kf)(\tau):=c_0+\sum\limits_{n=1}^\infty c(n)q^n,
\]
with $c_0=c_0^{(1)}$ and
\begin{gather}\label{eq:holproj}
\begin{aligned}
c(n)=\frac{(4\pi n)^{k-1}}{(k-2)!}\int_0^\infty a_f(n,y)e^{-4\pi ny}y^{k-2}dy
\end{aligned}
\end{gather}
for $n\geq 1$. For $\ell\notin\N_0$ we set as usual $\ell !:=\Gamma(\ell+1)$, where $\Gamma$ denotes Euler's Gamma function.
\end{definition}
\begin{proposition}\label{prop:pihol2}[\cite{GrZ86}, Proposition 5.1 and Proposition 6.2, \cite{IRR13}, Proposition 3.2 and Theorem 3.3] 
Let $f:\bbH\rightarrow \C$ be as in \Cref{def:pihol}.
\begin{enumerate}[(i)]
\item If $f$ is holomorphic, then $\pihol f=f$.
\item We have that $\pihol f\in M_k(\Gamma)$ if $k>2$ and $\pihol f$ is a quasi-modular form of weight $2$ (cf. \cite{KZ95}) if $k=2$. 
\end{enumerate}
\end{proposition}
We now collect some handy properties of the operator $\pihol$. For this we recall the following operators ($f$ as in \Cref{def:pihol}, $N\in\N$, and $r\in\{0,\dots ,N-1\}$, $\chi$ some generalized character)
\begin{align}
\label{eq:U} (f|U(N))(\tau)&=\sum\limits_{n\in\Z} a_f\left(Nn,\frac yN\right)q^n,\\
\label{eq:V} (f|V(N))(\tau)&=f(N\tau),\\
\label{eq:S} (f|S_{N,r})(\tau)&=\sum\limits_{\substack{n\in\Z\\n\equiv r\pmod{N}}} a_f(n,y)q^n\qquad\text{``sieving operator''},\\
\label{eq:tensor} (f\otimes\chi)(\tau)&\sum\limits_{n\in\Z} a_f(n,y)\chi(n)q^n,
\end{align}
which all map modular forms to modular forms, in general of different level.
\begin{lemma}\label{lem:proppihol}
Let $f:\bbH\rightarrow\C$ be a function as in \Cref{def:pihol}, $N\in\N$, and $r\in\{0,\dots ,N-1\}$. Then the following holds.
\begin{enumerate}[(i)]
\item The operator $\pihol$ commutes with all the operators $U(N),\,V(N)$, $S_{N,r}$, and $\otimes\chi$.
\item If $f$ is modular of weight $k>2$ on $\Gamma\leq\SLZ$ then we have 
\[\langle f,g\rangle=\langle\pihol(f),g\rangle,\]
where $\langle\cdot,\cdot\rangle$ denotes the Petersson scalar product, for every cusp form $g\in S_k(\Gamma)$.
\end{enumerate}
\end{lemma}
\begin{proof}
Assertion $(i)$ is obvious from the definition and for assertion $(ii)$ we refer to \cite{GrZ86}, Proposition 5.1.
\end{proof}
For the rest of this section, let $f\in\mathcal{H}_k(N)$ with a Fourier expansion as in \Cref{prop:split} and $g\in M_\ell(N)$ ($k,\ell\in\tfrac 12\N$, $k,\ell\neq 1$) with
\[g(\tau)=\sum\limits_{n=0}^\infty a_g(n)q^n\]
(we ignore characters for the moment) such that $k+\ell\in\N$, $k+\ell\geq 2$, and $[f,g]_\nu$ fulfills the conditions in \Cref{def:pihol}. This is the case for example if $f\in\calS_k$ and $f^+\cdot g$ is holomorphic at the cusps (cf. \cite{IRR13}, Theorem 3.5). 

Following and slightly extending \cite{ZagNotes}, we find an explicit formula for the Fourier coefficients of 
\begin{gather}\label{eq:piholRCB}
\pihol([f,g]_\nu)=[f^+,g]_\nu+\frac{(4\pi)^{1-k}}{k-1}\overline{c_f^-(0)}\pihol([y^{1-k},g]_\nu)+\pihol([f^-,g]_\nu).
\end{gather}
\begin{lemma}\label{lem:piholy}
We have 
\[\frac{(4\pi)^{1-k}}{k-1}\pihol([y^{1-k},g]_\nu)=\kappa\sum\limits_{n=0}^\infty n^{k+\nu-1}a_g(n)q^n\]
where $\kappa$ depends only on $k,\ell,\nu$. To be precise,
\begin{gather}\label{eq:kappa}
\begin{aligned}
\kappa=\kappa(k,\ell,\nu)&=\frac{1}{(k+\ell+2\nu-2)!(k-1)}\sum\limits_{\mu=0}^\nu\left[\frac{\Gamma(2-k)\Gamma(\ell+2\nu-\mu)}{\Gamma(2-k-\mu)}\right.\\
                         & \qquad\qquad\left.\times{{k+\nu-1} \choose \nu-\mu} {{\ell+\nu-1} \choose \mu}\right].
                         \end{aligned}
                         \end{gather}
\end{lemma}
\begin{proof}
It holds that 
\begin{align*}
D^\mu(y^{1-k})&=\frac{\Gamma(2-k)}{\Gamma(2-k-\mu)}\left(-\frac{1}{4\pi}\right)^\mu y^{1-k-\mu}
\end{align*}
and 
\[(D^{\nu-\mu}g)(\tau)=\sum\limits_{n=0}^\infty n^{\nu-\mu}a_g(n)q^n.\]

Thus the $n$th coefficient of $[y^{1-k},g]_\nu$ equals
\[a_g(n)\sum\limits_{\mu=0}^\nu (-1)^\mu {{k+\nu-1} \choose \nu-\mu} {{\ell+\nu-1} \choose \mu} \frac{\Gamma(2-k)}{\Gamma(2-k-\mu)}\left(-\frac{1}{4\pi}\right)^\mu y^{1-k-\mu}n^{\nu-\mu}.\]
 
We also calculate
\[\int\limits_0^\infty e^{-4\pi ny}y^{\ell+2\nu-\mu-1}=\left(\frac{1}{4\pi n}\right)^{\ell+2\nu-\mu}\Gamma(\ell+2\nu-\mu),\]
thus we get the following expression for the $n$th coefficient of $\pihol([y^{1-k},g]_\nu)$,
\begin{align*}
\frac{(4\pi n)^{k-1}}{(k+\ell+2\nu-2)!}n^\nu a_g(n)\sum\limits_{\mu=0}^\nu {{k+\nu-1} \choose \nu-\mu} {{\ell+\nu-1} \choose \mu} \frac{\Gamma(2-k)\Gamma(\ell+2\nu-\mu)}{\Gamma(2-k-\mu)},
\end{align*}
which is what we claimed.
\end{proof}
The third summand in \eqref{eq:piholRCB} requires a little more work. First we compute the derivatives of $f^-$.
\begin{lemma}\label{lem:Dmufm}
It holds that
\[(D^\mu f^-)(\tau)=(-1)^{\mu}\frac{\Gamma(1-k)}{\Gamma(1-k-\mu)}\sum\limits_{n=1}^\infty n^{k+\mu-1}\overline{c_f^-(n)}\Gamma(1-k-\mu;4\pi ny)q^{-n}\]
\end{lemma}
\begin{proof}
We write $f^-$ as
\[f^-(\tau)=\sum\limits_{n=1}^\infty n^{k-1}\overline{c_f^-(n)}\Gamma^*(1-k,4\pi ny)\overline{q}^{-n},\]
where $\Gamma^*(\alpha;x):=e^x\Gamma(\alpha;x)$. It is easy to check that
\[\frac{d^\mu}{dx^\mu}\Gamma^*(\alpha;x)=\frac{\Gamma(\alpha)}{\Gamma(\alpha-\mu)}\Gamma^*(\alpha-\mu;x).\]
Therefore we have
\begin{align*}
(D^\mu f^-)(\tau)&=\left(\frac{1}{2\pi i}\right)^\mu\sum\limits_{n=1}^\infty n^{k-1}\overline{c_f^-(n)}\left(\frac{4\pi n}{2i}\right)^\mu \frac{\Gamma(1-k)}{\Gamma(1-k-\mu)}\\
                &\qquad\qquad\qquad\qquad\qquad\qquad\times\Gamma^*\left(1-k-\mu;4\pi n\frac{1}{2i}(\tau-\overline{\tau})\right)\overline{q}^{-n}\\
                &=(-1)^{\mu}\frac{\Gamma(1-k)}{\Gamma(1-k-\mu)}\sum\limits_{n=1}^\infty n^{k+\mu-1}\overline{c_f^-(n)}\Gamma(1-k-\mu;4\pi ny)q^{-n},
\end{align*}
which is what we have claimed.
\end{proof}
With this we can give an expression for the holomorphic projection of Rankin-Cohen brackets. We define the following homogeneous polynomials $P_{a,b}(X,Y)\in\C[X,Y]$ of degree $a-2$.
\begin{gather}\label{eq:P}
P_{a,b}(X,Y):=\sum\limits_{j=0}^{a-2} {{j+b-2} \choose {j}} X^j(X+Y)^{a-j-2}.
\end{gather}
\begin{theorem}
Assuming that the coefficients $c_f^-(n)$ and $a_g(n)$ grow sufficiently moderately, i.e. the integral defining $\pihol([f^-,g]_\nu)$ is absolutely convergent, the holomorphic projection of $f^-$ and $g$ is given by
\[\pihol([f^-,g]_\nu)=\sum\limits_{r=1}^\infty b(r)q^r,\]
whereas $b(r)$ is given by
\begin{gather}\label{eq:piholfm}
\begin{aligned}
b(r)=&-\Gamma(1-k)\sum\limits_{m-n=r}\sum\limits_{\mu=0}^\nu {{k+\nu-1} \choose {\nu-\mu}}{{\ell+\nu-1} \choose \mu}m^{\nu-\mu}a_g(m)\overline{c_f^-(n)}\\
     &\qquad\qquad\times\left(m^{\mu-2\nu-\ell+1}P_{k+\ell+2\nu,2-k-\mu}(r,n)-n^{k+\mu-1} \right)
\end{aligned}
\end{gather}
\end{theorem}
\begin{proof}
By \Cref{lem:Dmufm} we see that
\[[f^-,g]_\nu(\tau)=\sum\limits_{r\in\Z} b(r,y)q^r,\]
where 
\begin{align*}
b(r,y)=&\sum\limits_{m-n=r}\sum\limits_{\mu=0}^\nu {{k+\nu-1} \choose \nu-\mu} {{\ell+\nu-1} \choose \mu}\\
       &\qquad\qquad\qquad\times\frac{\Gamma(1-k)}{\Gamma(1-k-\mu)}n^{k+\mu-1}\overline{c_f^-(n)}\Gamma(1-k-\mu;4\pi ny)a_g(m)m^{\nu-\mu}.
\end{align*}
Applying $\pihol$ then yields
\[\pihol([f^-,g]_\nu)=\sum\limits_{r=1}^\infty b(r)q^r,\]
with 
\begin{align*}
b(r)=&\frac{(4\pi r)^{k+\ell+2\nu-1}}{(k+\ell+2\nu-2)!}\sum\limits_{m-n=r}\sum\limits_{\mu=0}^\nu {{k+\nu-1} \choose \nu-\mu} {{\ell+\nu-1} \choose \mu}\\
     &\qquad\qquad\qquad\times\frac{\Gamma(1-k)}{\Gamma(1-k-\mu)}n^{k+\mu-1}\overline{c_f^-(n)}a_g(m)m^{\nu-\mu}\\
     &\qquad\qquad\qquad\times\int\limits_0^\infty \Gamma(1-k-\mu;4\pi ny)e^{-4\pi ry}y^{k+\ell+2\nu-2}dy.
\end{align*}
The interchanging of integration and summation is justified by the assumption on the growth of the Fourier coefficients and applying the Theorem of Fubini-Tonelli.

The calculation of the integral $\int\limits_0^\infty \Gamma(1-k-\mu;4\pi ny)e^{-4\pi ry}y^{k+\ell+2\nu-2}dy$ carried out in the following lemma implies the claim.
\end{proof}
\begin{lemma}
The following identity holds true.
\begin{align*}
I:&=\int\limits_0^\infty \Gamma(1-k-\mu;4\pi ny)e^{-4\pi ry}y^{k+\ell+2\nu-2}dy\\
  &=-(4\pi)^{1-(k+\ell+2\nu)}n^{1-k-\mu}\frac{\Gamma(1-k-\mu)(k+\ell+2\nu-2)!}{r^{k+\ell+2\nu-1}}\\
  &\qquad\qquad\qquad\times\left[(r+n)^{\mu-\ell-2\nu+1}P_{k+\ell+2\nu,2-k-\mu}(r,n)-n^{k+\mu-1} \right].
\end{align*}
\end{lemma}
\begin{proof}
Written as double integral, $I$ equals
\[\int\limits_0^\infty\int\limits_{4\pi ny}^\infty e^{-t}t^{-k-\mu}e^{4\pi ry}y^{k+\ell+2\nu-2} dtdy.\]
Substituting $4\pi nyt'=t$ this equals
\begin{align*}
 &(4\pi n)^{1-k-\mu}\int\limits_0^\infty\int\limits_1^\infty e^{-4\pi y(r+nt)}t^{-k-\mu}y^{\ell+2\nu-1}dtdy\\ 
=&(4\pi n)^{1-k-\mu}\int\limits_1^\infty t^{-k-\mu}\int\limits_0^\infty e^{-4\pi y(r+nt)}y^{\ell+2\nu-1}dtdy.
\end{align*}
For the inner integral we substitute $4\pi (r+nt)=y'$, hence $dy=(4\pi (r+nt))^{-1}dy'$, and find that
\[I=(4\pi)^{1-(k+\ell+2\nu)}n^{1-k-\mu}\Gamma(\ell+2\nu-\mu)\int\limits_1^\infty t^{-k-\mu}(r+nt)^{\mu-\ell-2\nu}dt,\]
which after yet another substitution $t=\tfrac{1}{t'}$, hence $dt=-\tfrac{1}{t'^2}dt'$, simplifies to
\begin{align*}
 &(4\pi)^{1-(k+\ell+2\nu)}n^{1-k-\mu}\Gamma(\ell+2\nu-\mu)\int\limits_0^1 t^{k+\mu-2}\left(r+\frac nt\right)^{\mu-\ell-2\nu}dt\\
=&(4\pi)^{1-(k+\ell+2\nu)}n^{1-k-\mu}\Gamma(\ell+2\nu-\mu)\int\limits_0^1 t^{k+\ell+2\nu-2}(rt+n)^{\mu-\ell-2\nu}dt.
\end{align*}
Since $k+\ell$ is an integer by assumption, we note
\begin{gather}\label{eq:partialr}
\frac{\partial^{k+\ell+2\nu-2}}{\partial r^{k+\ell+2\nu-2}}(rt+n)^{k+\mu-2}=\frac{\Gamma(k+\mu-1)}{\Gamma(\mu-\ell-2\nu+1)}t^{k+\ell+2\nu-2}(rt+n)^{\mu-\ell-2\nu},
\end{gather}
and thus
\begin{align*}
 & \int\limits_0^1 t^{k+\ell+2\nu-2}(rt+n)^{\mu-\ell-2\nu}dt\\
=&\frac{\Gamma(\mu-\ell-2\nu+1)}{\Gamma(k+\mu-1)}\frac{\partial^{k+\ell+2\nu-2}}{\partial r^{k+\ell+2\nu-2}}\int\limits_0^1(rt+n)^{k+\mu-2}dt\\
=&\frac{\Gamma(\mu-\ell-2\nu+1)}{\Gamma(k+\mu-1)}\frac{\partial^{k+\ell+2\nu-2}}{\partial r^{k+\ell+2\nu-2}} \left[\frac{1}{k+\mu-1}\frac 1r\left((r+n)^{k+\mu-1} -n^{k-\mu-1}\right)\right].
\end{align*}
By using \eqref{eq:partialr} and simplifying the expressions involving Gamma functions using the reflection formula \eqref{eq:Gamma}, the assertion follows after a little computation.
\end{proof}
For a more detailed proof of this, we refer the reader to Lemma V.1.7. in \cite{MertPhD}.
\section{The case of weight $\left(\tfrac 32,\tfrac 12\right)$}\label{sec:weight}
In this section, we work out the explicit formula for weight $\left(\tfrac 32,\tfrac 12\right)$ mixed mock modular form of $\nu$th type which becomes surprisingly simple. For this we will need two preparatory lemmas.
\begin{lemma}\label{lem:PXY}
For $b\neq 1,2$, the polynomial $P_{a,b}(X,Y)$ from \eqref{eq:P} fulfills the following identities.
\begin{align*}
P_{a,b}(X,Y)&=\sum\limits_{j=0}^{a-2} {{a+b-3} \choose j}X^jY^{a-2-j}\\
            &=\sum\limits_{j=0}^{a-2} {{a+b-3} \choose {a-2-j}}{{j+b-2} \choose j} (X+Y)^{a-2-j}(-Y)^j.
            \end{align*}
\end{lemma}
\begin{proof}
We first prove the first identity by induction on $a$. 

For $a=2$ we have $1=1$.

Assume the claim to be true for one $a\geq 2$, then we get for $a+1$:
{\allowdisplaybreaks
\begin{align*}
P_{a+1,b}(X,Y)&=\sum\limits_{j=0}^{a-1} {{j+b-2} \choose j} X^j(X+Y)^{a-2-j+1}\\
              &\overset{IV}{=}(x+y)\sum\limits_{j=0}^{a-2} {{a+b-3} \choose j}X^jY^{a-2-j}+{{a+b-3} \choose {a-1}}X^{a-1}\\
              &=\sum\limits_{j=0}^{a-2} {{a+b-3} \choose j}X^{j+1}Y^{a-2-j}+\sum\limits_{j=0}^{a-2} {{a+b-3} \choose j}X^jY^{a+1-j-2}\\
              &\qquad\qquad\qquad +{{a+b-3} \choose a-1}X^{a-1}\\
              &=\sum\limits_{j=1}^{a-1} {{a+b-3} \choose {j-1}}X^{j}Y^{a-1-j}+\sum\limits_{j=0}^{a-2} {{a+b-3} \choose j}X^jY^{a-j-1}\\
              &\qquad\quad\qquad +{{a+b-3} \choose a-1}X^{a-1}\\
              &=\sum\limits_{j=1}^{a-2} \left[{{a+b-3} \choose {j-1}}+{{a+b-3} \choose j}\right]X^{j}Y^{a-1-j}\\
              &\qquad\qquad\qquad +Y^{a-1}+\left[{{a+b-3} \choose a-1}+{{a+b-3} \choose {a-2}}\right]X^{a-1}\\
              &=\sum\limits_{j=0}^{a-1} {{a+b-2} \choose {j}}X^{j}Y^{a-1-j},
\end{align*}  
}
hence the first equation follows.

Now we show that
\[P_{a,b}(X,Y)=\sum\limits_{j=0}^{a-2} {{a+b-3} \choose {a-2-j}}{{j+b-2} \choose j} (X+Y)^{a-2-j}(-Y)^j,\]
again by induction on $a$. 

The case $a=2$ again yields $1=1$.

Suppose the equality to show holds for some $a\geq 2$ (Z:=X+Y):
\[\sum\limits_{j=0}^{a-2} {{j+b-2} \choose j}(Z-Y)^jZ^{a-2-j}=\sum\limits_{j=0}^{a-2}{{a+b-3} \choose {a-2-j}}{{j+b-2} \choose j} Z^{a-2-j}(-Y)^j.\]
Integration with respect to $Y$ gives ($C$ some constant to be determined later):
\begin{align*}
&-\sum\limits_{j=0}^{a-2} {{j+b-2} \choose j}\frac{1}{j+1}(Z-Y)^{j+1}Z^{a-2-j}\\
&\qquad\qquad\qquad =C+\sum\limits_{j=0}^{a-2}{{a+b-3} \choose {a-2-j}}{{j+b-2} \choose j} Z^{a-2-j}\frac{(-1)^j}{j+1}Y^{j+1}.
\end{align*}
Since by assumption we have $b\neq 2$, one easily sees that
\[\frac{1}{j+1}{{j+b-2} \choose j}=\frac{((j+1)+(b-1)-2)!}{(j+1)!(b-2)!}=\frac{1}{b-2}{{j+b-2} \choose {j+1}},\]
and therefore the above is equivalent to (replacing $b$ by $b+1$)
\begin{align*}
&\sum\limits_{j=1}^{a-1} {{j+b-2} \choose j}(Z-Y)^jZ^{a+1-j-2}\\
&\qquad\qquad\qquad =C'+\sum\limits_{j=1}^{a-1}{{a+1+b-3} \choose {a+1-2-j}}{{j+b-2} \choose j}Z^{a+1-j-2}(-1)^jY^j.
\end{align*}
If we let both sums start at $j=0$, then we just add constant terms in $Y$, thus comparison of the constant terms yields the assertion. Hence we have to prove 
\[Z^{a-1}\sum\limits_{j=0}^{a-1}{{j+b-2} \choose j}=Z^{a-1}{{a+b-2} \choose {a-1}},\]
which follows from the first identity we showed by plugging in $X=1$ and $Y=0$.
\end{proof}
\begin{lemma}\label{lem:binomial}
\begin{enumerate}[(i)]
\item The following identity is valid for $\nu>0$,
\[\sum\limits_{\mu=0}^\nu \frac{(-1)^{\mu}}{\mu-j+\frac 12}{{4\nu-2\mu-1} \choose {2(\nu-\mu),2\nu-\mu-1}}=2^{4\nu}(-1)^j \frac{(2\nu-j)!j!}{(2j)!(2(\nu-j)+1)!}.\]
\item For $\kappa$ from \eqref{eq:kappa} we have for all $\nu\geq 0$ that
\[\kappa\left(\frac 32,\frac 12,\nu\right)=2^{1-2\nu}\sqrt{\pi}{{2\nu} \choose \nu}.\]
\end{enumerate}
\end{lemma}
\begin{proof}
We first prove $(i)$. Write
\[\sum\limits_{\mu=0}^\nu \frac{(-1)^{\mu}}{\mu-j+\frac 12}{{4\nu-2\mu-1} \choose {2(\nu-\mu),2\nu-\mu-1}}=\sum\limits_{\mu=0}^\nu c_\mu\]

By a standard computation it is easy to see that $\sum c_\mu$ is a multiple of a hypergeometric series,
\[\sum\limits_{\mu=0}^\nu c_\mu=\frac{1}{-j+\frac 12}{{4\nu-1} \choose {2\nu}}{}_3F_2\left(\begin{matrix} -\nu,-j+\frac 12,-\nu+\frac{1}{2} \\ -j+\frac 32, -2\nu+\frac 12\end{matrix} ; 1\right),\]
which by the Pfaff-Saalsch\"utz identity \eqref{eq:Pfaff} equals
\begin{align*}
 &\frac{1}{-j+\frac 12}{{4\nu-1} \choose {2\nu}}\frac{(1)_\nu(\nu-j+1)_\nu}{\left(-j+\frac 32\right)_\nu\left(\nu+\frac 12\right)_\nu}\\
=&{{4\nu-1} \choose {2\nu}}\nu !\frac{(2\nu-j)!}{(\nu-j)!}\frac{\Gamma\left(-j+\frac 12\right)\Gamma\left(\nu+\frac 12\right)}{ \Gamma\left(\nu-j+\frac 32\right)\Gamma\left(2\nu+\frac 12\right)}.
\end{align*}
Simplifying this a little further yields the assertion.

The proof of $(ii)$ works in the very same way, so we omit it here. It is carried out in Lemma V.2.6 in \cite{MertPhD}.
\end{proof}
We can now prove the following result.
\begin{proposition}\label{prop:Pmn}
Let $r=m-n$. Then it holds that
\[\sum\limits_{\mu=0}^\nu {{\nu+\frac 12} \choose {\nu-\mu}}{{\nu-\frac 12} \choose \mu}\left(m^{\frac 12-\nu} P_{2\nu+2,\frac 12-\mu}(r,n)-n^{\frac 12+\mu}m^{\nu-\mu}\right)=2^{-2\nu}{{2\nu} \choose \nu}(m^{\frac 12}-n^{\frac 12})^{2\nu+1}.\]
\end{proposition}
\begin{proof}
For $\nu=0$, the identity is immediate, thus assume from now on $\nu\geq 1$.

By applying Legendre's duplication formula \eqref{eq:Legendre} several times we get
{\allowdisplaybreaks
\begin{align*}
 &{{\nu+\frac 12} \choose {\nu-\mu}}{{\nu-\frac 12} \choose \mu}=2^{-2\nu}{{2\nu} \choose \nu}{{2\nu+1} \choose {2\mu+1}}.
\end{align*}
}
On the one hand we have
\[(m^{\frac{1}{2}}-n^{\frac 12})^{2\nu+1}=\sum\limits_{\mu=0}^\nu {{2\nu+1} \choose {2\mu}}m^{\nu-\mu+\frac 12}n^\mu -\sum\limits_{\mu=0}^\nu{{2\nu+1} \choose {2\mu+1}}m^{\nu-\mu}n^{\mu+\frac 12},\]
on the other we see by \Cref{lem:PXY} that
\begin{align*}
 &\sum\limits_{\mu=0}^\nu {{2\nu+1} \choose {2\mu+1}}\left(m^{\frac 12-\nu} P_{2\nu+2,\frac 12-\mu}(r,n)-n^{\frac 12+\mu}m^{\nu-\mu}\right)\\
=&\sum\limits_{\mu=0}^\nu {{2\nu+1} \choose {2\mu+1}}\sum\limits_{j=0}^{2\nu} {{2\nu-\mu-\frac 12} \choose {2\nu-j}}{{j-\mu-\frac 32} \choose j}m^{\nu-j+\frac 12}(-n)^j\\
 &\qquad\qquad\qquad\qquad-\sum\limits_{\mu=0}^\nu {{2\nu+1} \choose {2\mu+1}}n^{\frac 12+\mu}m^{\nu-\mu}.
\end{align*}
Thus it remains to show the following identity.
\begin{gather}\label{eq:idsum}
\sum\limits_{\mu=0}^\nu {{2\nu+1} \choose {2\mu+1}}\sum\limits_{j=0}^{2\nu} {{2\nu-\mu-\frac 12} \choose {2\nu-j}}{{j-\mu-\frac 32} \choose j}m^{\nu-j+\frac 12}(-n)^j=\sum\limits_{\mu=0}^\nu {{2\nu+1} \choose {2\mu}}m^{\nu-\mu+\frac 12}n^\mu.
\end{gather}
Again, we first simplify the product of the binomial coefficients in the inner sum. By \eqref{eq:Legendre} and \eqref{eq:Gamma} we find that
{\allowdisplaybreaks
\begin{align*}
 &{{2\nu-\mu-\frac 12} \choose {2\nu-j}}{{j-\mu-\frac 32} \choose j}\\
=&\frac{(-1)^{\mu+1}}{j-\mu-\frac 12}2^{-4\nu}\frac{(4\nu-2\mu-1)!(2\mu+1)!}{(2\nu-\mu-1)!(2\nu-j)!j!\mu!}.
\end{align*}
}
Now we have a look at the left-hand side of \eqref{eq:idsum}.
\begin{align*}
 &\sum\limits_{\mu=0}^\nu {{2\nu+1} \choose {2\mu+1}}\sum\limits_{j=0}^{2\nu} {{2\nu-\mu-\frac 12} \choose {2\nu-j}}{{j-\mu-\frac 32} \choose j}m^{2\nu-j}(-n)^j\\
=&2^{-4\nu}\sum\limits_{j=0}^{2\nu}{{2\nu+1} \choose {2j}}\frac{(2j)!(2(\nu-j)+1)!}{2(\nu-j)!j!}m^{2\nu-j}(-n)^j\\
 &\qquad\qquad\qquad\times\sum\limits_{\mu=0}^\nu \frac{(-1)^{\mu+1}}{j-\mu-\frac 12} {{4\nu-2\mu-1} \choose {2(\nu-\mu),2\nu-\mu-1,\mu}}.
\end{align*}
By \Cref{lem:binomial}$(i)$ we see that \eqref{eq:idsum} is valid and so is our Proposition.
\end{proof}
From the preceding proposition we can now prove our first main result.
\begin{theorem}\label{theo:main32}
Let $f\in\mathcal{M}_{\frac 32}(\Gamma)$ and $g\in M_{\frac 12}(\Gamma)$ with $\Gamma=\Gamma_1(4N)$ for some $N\in\N$ and fix $\nu\in\N_0$. Then there is a finite linear combination $L_\nu^{f,g}$ of functions of the form
\begin{align*}
\Lambda_{s,t}^{\chi,\psi}(\tau;\nu)=&\sum\limits_{r=1}^\infty \left(2\sum\limits_{\substack{sm^2-tn^2=r \\ m,n\geq 1}} \chi(m)\overline{\psi(n)}(\sqrt{s}m-\sqrt{t}n)^{2\nu+1}\right)q^r\\
                                &\qquad\qquad\qquad +\overline{\psi(0)}\sum\limits_{r=1}^\infty \chi(r)(\sqrt{s}r)^{2\nu+1}q^{sr^2}
\end{align*}
where $s,t\in\N$ and $\chi,\psi$ are even characters of conductors $F(\chi)$ and $F(\psi)$ respectively with $sF(\chi)^2,tF(\psi)^2|N$, such that 
\[[f,g]_\nu+L_\nu^{f,g}\]
is a (holomorphic) quasi-modular form of weight $2$ if $\nu=0$ or otherwise a holomorphic modular form of weight $2\nu+2$.
\end{theorem}
\begin{proof}
By assumption both $(\xi f)$ and $g$ are holomorphic modular forms of weight $\tfrac 12$, hence by the Theorem of Serre-Stark (\cite{SS77}, Theorem A) linear combinations of unary theta functions, i.e. functions of the form
\[\vartheta_{s,\chi} (\tau)=\sum\limits_{n\in\Z} \chi(n)q^{sn^2},\]
where $s,\chi$ fulfill the conditions given in our Theorem. Thus we may assume without loss of generality that $(\xi f)$ and $g$ are in fact unary theta series, say $(\xi f)=\vartheta_{t,\psi}$ and $g=\vartheta_{s,\chi}$. By formally using \Cref{prop:Pmn}, \Cref{lem:piholy}, and \Cref{lem:binomial} $(ii)$ inside \eqref{eq:piholfm} we immediately get up to a constant factor of 
\begin{gather}\label{eq:const}
4^{1-\nu}{{2\nu} \choose \nu}\sqrt{\pi}.
\end{gather}
the formula for $\Lambda_{s,t}^{\chi,\psi}$ that we stated in the Theorem.

To complete the proof, we have to check that the sum
\[\sum\limits_{\substack{sm^2-tn^2=r \\ m,n\geq 1}} \chi(m)\overline{\psi(n)}(\sqrt{s}m-\sqrt{t}n)^{2\nu+1}\]
for the coefficients converges. If $s$ and $t$ are both perfect squares, the sum is actually finite, since then $sm^2-tn^2$ factors in the rational integers and thus each summand is a power of a divisor of $r$ of which there are but finitely many. Let us now assume for simplicity that $s=1$ and $t$ is square-free, the general case works essentially in the same way. The Pell type equation 
\begin{gather}\label{eq:Pell}
m^2-tn^2=r
\end{gather}
is well-known to have at most finitely many fundamental integer solutions: if $(a,b)$ is the fundamental solution of $a^2-tb^2=1$ (i.e. $\eps:=a-\sqrt{t}b>1$), then there is a solution $(m_0,n_0)$ of \eqref{eq:Pell} such that $|m_0|\leq\tfrac{\eps+1}{2\sqrt{\eps}}\sqrt{r}$ and $n_0\leq\sqrt{\tfrac{m^2-r}{t}}$. In particular, there are only finitely many such so-called fundamental solutions. Then all solutions $(m,n)$ in $\Z$ of \eqref{eq:Pell} satisfy
\[m+\sqrt{t}n=\pm(m_0+\sqrt{t}n_0)\cdot \eps^k\]
for one $k\in\Z$. We are interested in solutions $(m,n)$ in $\N$. It is plain that such a solution exists, provided that there is one in $\Z$. Furthermore, we see immediately, that the power of the fundamental unit $\eps$ has to be negative to parametrize all possible solutions. In particular this means that
\begin{align*}
&\left|\sum\limits_{\substack{m^2-tn^2=r \\ m,n\geq 1}} \chi(m)\overline{\psi(n)}(m-\sqrt{t}n)^{2\nu+1}\right|\\
\leq & \sum\limits_{\substack{m^2-tn^2=r \\ m,n\geq 1}} (m-\sqrt{t}n)^{2\nu+1}\\
=&\sum\limits_{m_0,n_0} (m_0-\sqrt{t}n_0)^{2\nu+1}\sum\limits_{k=0}^\infty \eps^{-k(2\nu+1)}<\infty
\end{align*}
because $\sum\limits_{k=0}^\infty \eps^{-k}$ is a geometric series and the set of possible $(m_0,n_0)$ is finite. This completes the proof.
\end{proof}
Obviously, $\Lambda_{s,t}$ from \Cref{theo:main32}  is an indefinite theta function with some polynomial factor. For simplicity, we assume $s$ and $t$ to be coprime from now on. It is then easy to see (cf. \Cref{sec:Ex} and \cite{Mert13}) that the function $\Lambda_{s,t}^{\chi,\psi}$ is a linear combination of derivatives of Appell-Lerch sums provided $s$ and $t$ are both perfect squares. If at least one of $s$ and $t$ is not a perfect square, one supposedly needs a certain generalization of Appell-Lerch sums of which $\Lambda_{s,t}^{\chi,\psi}$ is a derivative. The author plans to address this question in a forthcoming publication.

To conclude this section, we mention a nice structural corollary of \Cref{theo:main32}.
\begin{corollary}
With the notation from \Cref{theo:main32} the following is true. 

The equivalence classes $\Lambda_{s,t}^{\chi,\psi}+M_{2\nu+2}^!(\Gamma_1(N))$ generate the $\C$-vector space 
\[[\M^{mock}_{\frac 32}(\Gamma),M_\frac{1}{2}(\Gamma)]_\nu/M_{2\nu+2}^!(\Gamma)\]
of all $\nu$th order Rankin-Cohen brackets of mock modular forms of weight $\tfrac 32$ with holomorphic shadow with weight $\tfrac 12$ modular forms modulo the space of weakly holomorphic modular forms of weight $2\nu+2$.
\end{corollary}
\begin{proof}
This is just another way to state \Cref{theo:main32}.
\end{proof}
\section{Mock Theta Functions}\label{sec:MockTheta}
In this section we obtain analoguous results as in \Cref{sec:weight}. So let us assume that $f\in\mathcal{S}_{\frac 12}$ is a mock theta function, i.e. $\xi f$ is a (linear combination of) weight $\tfrac 32$ unary theta series $\theta_{\chi,s}$ as defined in \eqref{eq:theta} (which are always cusp forms), and $g$ also a (linear combination of) such functions. The calculations in this section are essentially the same ones as in \Cref{sec:weight}.
\begin{lemma}\label{lem:binomial12}
For all $\nu\in\N_0$ we have that
\[\sum\limits_{\mu=0}^\nu \frac{(-1)^\mu}{(2(j-\mu)+1}{{4\nu-2\mu+1} \choose {2(\nu-\mu)+1,2\nu-\mu}}=(-1)^j2^{4\nu}\frac{(2\nu-j)!j!}{(2(\nu-j))!(2j+1)!}.\]
\end{lemma}
\begin{proof}
As in \Cref{lem:binomial}, we write the left hand side as a hypergeometric series. Call the summand $c_\mu$, then see that
\begin{align*}
\sum\limits_{\mu=0}^\nu c_\mu=&\frac{1}{2j+1}{{4\nu+1} \choose {2\nu+1}}{}_3F_2\left(\begin{matrix} -\nu,-j-\frac 12,-\nu-\frac 12 \\ -j+\frac 12,-2\nu-\frac 12\end{matrix};1\right)\\
                             \overset{\eqref{eq:Pfaff}}{=}&\frac{1}{2j+1}{{4\nu+1} \choose {2\nu+1}}\frac{(1)_\nu(\nu-j+1)_\nu}{\left(-j+\frac 12\right)_\nu\left(\nu+\frac 32\right)_\nu}\\
                             =&\frac{1}{2j+1}{{4\nu+1} \choose {2\nu+1}}\nu!\frac{(2\nu-j)!}{(\nu-j)!}\frac{\Gamma\left(\nu+\frac 32\right)\Gamma\left(-j+\frac 12\right)}{\Gamma\left(\nu-j+\frac 12\right)\Gamma\left(2\nu+\frac 32\right)}
\end{align*}
As before, the result is now easily obtained by simplifying this expression using the duplication formula \eqref{eq:Legendre} and the reflection formula \eqref{eq:Gamma}.
\end{proof}
\begin{proposition}\label{prop:Pmn12}
The following identity holds true for all $\nu\geq 0$ and $r:=m-n$.
\begin{gather}\label{eq:idsum12}
\begin{aligned}
&\sum\limits_{\mu=0}^\nu {{\nu-\frac 12} \choose {\nu-\mu}}{{\nu+\frac 12} \choose \mu}\left(m^{-\nu-\frac 12}P_{2\nu+2,\frac 32-\mu}(r,n)-n^{\mu-\frac 12}m^{\nu-\mu}\right) \\
&\qquad\qquad\qquad =-2^{-2\nu}{{2\nu}\choose \nu}(mn)^{-\frac 12}\left(m^\frac{1}{2}-n^\frac 12\right)^{2\nu+1}
\end{aligned}
\end{gather}
\end{proposition}
\begin{proof}
The assertion is obvious for $\nu=0$, thus suppose $\nu\geq 1$.

From the proof of \Cref{prop:Pmn} we immediately get that
\begin{align*}
{{\nu-\frac 12} \choose {\nu-\mu}}{{\nu+\frac 12} \choose \mu}=2^{-2\nu}{{2\nu}\choose \nu}{{2\nu+1} \choose {2\mu}}.
                                                              \end{align*}
Thus using \Cref{lem:PXY}, we get that the left-hand side of \eqref{eq:idsum12} equals
\begin{align*}
2^{-2\nu}{{2\nu}\choose \nu}\left(\sum\limits_{\mu=0}^\nu {{2\nu+1} \choose {2\mu}}  \left[\sum\limits_{j=0}^{2\nu} (-1)^j{{2\nu-\mu+\frac 12} \choose {2\nu-j}}{{j-\mu-\frac 12} \choose j}m^{\nu-j-\frac 12}n^j\right]\right.\\
\left.  -    \sum\limits_{\mu=0}^\nu   {{2\nu+1} \choose {2\mu}} m^{\nu-\mu}n^{\mu-\frac 12}\right)
\end{align*}
while the right-hand side is given by
\[-2^{-2\nu}{{2\nu}\choose \nu}\left(\sum\limits_{\mu=0}^\nu {{2\nu+1} \choose {2\mu}}m^{\nu-\mu+\frac 12}n^{\mu-\frac 12}-\sum\limits_{\mu=0}^\nu {{2\nu+1} \choose {2\mu+1}} m^{\nu-\mu-\frac 12}n^{\mu} \right).\]
Hence we just have to show that
\[\sum\limits_{\mu=0}^\nu {{2\nu+1} \choose {2\mu}}  \left[\sum\limits_{j=0}^{2\nu} (-1)^j{{2\nu-\mu+\frac 12} \choose {2\nu-j}}{{j-\mu-\frac 12} \choose j}m^{\nu-j-\frac 12}n^j\right]=\sum\limits_{\mu=0}^\nu {{2\nu+1} \choose {2\mu+1}} m^{\nu-\mu-\frac 12}n^{\mu} .\]
The product of the binomial coefficients in the inner sum can be simplified as in the proof of \Cref{lem:binomial}:
\begin{align*}
{{2\nu-\mu+\frac 12} \choose {2\nu-j}}{{j-\mu-\frac 12} \choose j}=\frac{\left(4\nu-2\mu+1\right)}{\left(2(j-\mu)+1\right)}(-1)^{\mu} 2^{-4\nu+1}\frac{(4\nu-2\mu-1)!(2\mu)!}{(2\nu-\mu-1)!(2\nu-j)!j!\mu!},
\end{align*}
such that we have for the left-hand side the following,
\begin{align*}
 &2^{-4\nu}\sum\limits_{\mu=0}^\nu \sum\limits_{j=0}^{2\nu} (-1)^j \frac{(2\nu+1)!\left(4\nu-2\mu+1\right)!}{(2(\nu-\mu)+1)!(2\nu-\mu)!\left(2(j-\mu)+1\right)(2\nu-j)!j!\mu!}(-1)^\mu m^{\nu-j-\frac 12}n^j\\
=&2^{-4\nu}\sum\limits_{j=0}^{2\nu}(-1)^j\frac{(2\nu+1)!}{(2\nu-j)!j!}m^{\nu-j-\frac 12}n^j \sum\limits_{\mu=0}^\nu \frac{(-1)^\mu}{(2(j-\mu)+1}{{4\nu-2\mu+1} \choose {2(\nu-\mu)+1,2\nu-\mu}},
\end{align*}
thus by \Cref{lem:binomial12} the assertion follows.
\end{proof}
This implies our second main theorem.
\begin{theorem}\label{theo:main12}
Let $f\in\calS_{\frac 12}(\Gamma)$ be a mock theta function and $g\in\calS_{\tfrac 32}(\Gamma)$ be a theta function of weight $\tfrac 32$, where $\Gamma=\Gamma_1(4N)$ for some $N\in\N$ and let $\nu$ be a fixed non-negative integer. Then there is a finite linear combination $D_\nu^{f,g}$ of functions of the form
\begin{align*}
\Delta_{s,t}^{\chi,\psi}(\tau;\nu)=&2\sum\limits_{r=1}^\infty \left(\sum\limits_{\substack{sm^2-tn^2=r \\ m,n\geq 1}} \chi(m)\overline{\psi(n)}(\sqrt{s}m-\sqrt{t}n)^{2\nu+1}\right)q^r,
\end{align*}
where $s,t\in\N$ and $\chi,\psi$ are odd characters of conductors $F(\chi)$ and $F(\psi)$ respectively with $sF(\chi)^2,tF(\psi)^2|N$, such that 
\[[f,g]_\nu+D_\nu^{f,g}\]
is a (holomorphic) quasi-modular form of weight $2$ if $\nu=0$ or otherwise a holomorphic modular form of weight $2\nu+2$.
\end{theorem}
\begin{proof}
Follows immediately from \ref{eq:piholfm} using \Cref{prop:Pmn12} and the fact that both $\xi f$ and $g$ are by assumption linear combinations of weight $\tfrac 32$ theta functions of the form
\begin{gather}\label{eq:theta}
\theta_{\chi,s}=\sum\limits_{n\in\Z} n\chi(n)q^{sn^2}
\end{gather}
with $\chi$ and $s$ as described in the Theorem. The application of \eqref{eq:piholfm} can be justified just as in the proof of \Cref{theo:main32}.
\end{proof}
\begin{remark}
\begin{enumerate}[(i)]
\item Up to the difference in the required parity of the characters, the functions $\Lambda_{s,t}^{\chi,\psi}$ from \Cref{theo:main32} and $\Delta_{s,t}^{\chi,\psi}$ from \Cref{theo:main12} are exactly the same. Note that for an odd character $\psi$ we always have $\psi(0)=0$.
\item Note that it is a real restriction to require $f$ to be a completed mock theta function and $g$ to be a weight $\tfrac 32$ unary theta function in the statement of \Cref{theo:main12} since there is no analogue of the theorem of Serre-Stark in weight $\tfrac 32$, i.e. there are modular forms of that weight, which are NOT linear combinations of theta functions.
\end{enumerate}
\end{remark}
\begin{corollary}
With the notation from \Cref{theo:main12} the following is true. 

The equivalence classes $\Delta_{s,t}^{\chi,\psi}+M_{2\nu+2}^!(\Gamma_1(N))$ generate the $\C$-vector space 
\[[\calS^{mock-\vartheta}_{\frac 12}(\Gamma),\calS^\vartheta_\frac{3}{2}(\Gamma)]_\nu/M_{2\nu+2}^!(\Gamma)\]
of all $\nu$th order Rankin-Cohen brackets of mock theta functions with weight $\tfrac 32$ theta functions modulo the space of weakly holomorphic modular forms of weight $2\nu+2$.
\end{corollary}
\section{Examples}\label{sec:Ex}
\subsection{Trace Formulas}\label{sec:Trace}
Let $\calH$ denote the generating function of Hurwitz class numbers as defined in the introduction and $\vartheta(\tau)=\sum\limits_{n\in\Z} q^{n^2}$. Then we have the following result due to D. Zagier (\cite{HZ76}, Theorem 2 of Chapter 2)
\begin{theorem}[Zagier, 1976]
Let 
\[\calR(\tau):=\frac{1+i}{16\pi}\int\limits_{-\overline{\tau}}^{i\infty} \frac{\vartheta(z)}{(z+\tau)^{\frac 32}}dz=\frac{1}{8\pi\sqrt{y}}+\frac{1}{4\sqrt{\pi}}\sum\limits_{n=1}^\infty n\Gamma\left(-\frac 12;4\pi n^2y\right)q^{-n^2}.\]
Then the function $\widehat{\calH}=\calH+\calR$ is a harmonic Maa{\ss} form of weight $\tfrac 32$ on $\Gamma_0(4)$ with shadow $\xi\widehat{\calH}=\tfrac{1}{8\sqrt{\pi}}\vartheta$.
\end{theorem}
Consider the function $[\calH,\vartheta]_\nu$ for some $\nu\geq 0$. \Cref{theo:main32}, Equation \eqref{eq:const} and \Cref{prop:pihol2} tell us that
\[[\calH,\vartheta]_\nu+2^{-2\nu-1}{{2\nu} \choose \nu}\Lambda\]
with $\Lambda(\tau;\nu)=\Lambda_{1,1}^{1,1}(\tau;\nu)$ as in \Cref{theo:main32} is a quasi-modular form of weight $2$ for $ \nu=0$ and a holomorphic cusp form of weight $2\nu+2$ otherwise, both on the group $\Gamma_0(4)$. It is easy to see that 
\[(\Lambda\vert U(4))(\tau;\nu)=2^{2\nu+1}\sum\limits_{n=1}^\infty 2\lambda_{2\nu+1}(n)q^n\]
and 
\[(\Lambda\vert S_{2,1})(\tau;\nu)=2\sum\limits_{n\text{ odd}} \lambda_{2\nu+1}(n)q^n.\]

By \cite[Theorem 6.1]{Coh75} and \cite[Lemma 3.2]{Mert13} one has the formal identity
\[S_f^1(\tau;X)=\sum\limits_{n=0}^\infty\left(\sum\limits_{s\in\Z} \frac{a(n-s^2)}{1-2sX+nX^2}\right)q^n=\sum\limits_{\nu=0}^\infty \frac{4^\nu}{{{2\nu} \choose \nu}}[f,\vartheta]_\nu(\tau),\]
where $f(\tau)=\sum\limits_{n=0}^\infty a(n)q^n$ is a modular form of weight $\tfrac 32$. From there we can deduce that also the following is true
\[\left(S_f^1|U(4)\right)(\tau;X)=\sum\limits_{n=0}^\infty\left(\sum\limits_{s\in\Z} \frac{a(4n-s^2)}{1-s(2X)+n(2X)^2}\right)q^n=\sum\limits_{\nu=0}^\infty \frac{4^\nu}{{{2\nu} \choose \nu}}([f,\vartheta]_\nu|U(4))(\tau).\]
Hence the function ($N\in\{1,4\}$)
\[C^{(N)}(\tau)=\sum\limits_{n=1}^\infty c_\nu^{(N)}(n) q^n\]
with
\begin{align}
\label{ES}c_\nu^{(1)}(n)=&\sum\limits_{s\in\Z} g^{(1)}_\nu(s,n)H(4n-s^2)+2\lambda_{2\nu+1}(n)\\
\label{Coh}c_\nu^{(4)}(n)=&\begin{cases} \sum\limits_{s\in\Z} g^{(4)}_\nu(s,n)H(n-s^2)+\lambda_{2\nu+1}(n) & \text{for }n\text{ odd}\\ 0  & \text{otherwise,}\end{cases}
\end{align} 
where $g^{(1)}_\nu(s,n)$ (resp. $g^{(4)}_\nu$) is the coefficient of $X^{2\nu}$ in the Taylor expansion of $\frac{1}{1-sX+nX^2}$ (resp. $\frac{1}{1-2sX+nX^2}$), is a holomorphic modular form of weight $2\nu+2$ on $\Gamma_0(4)$. Since the constant term in the Fourier expansion vanishes for $\nu>0$ and the Rankin-Cohen brackets interchange with the slash operator, we see that in this case we are actually dealing with cusp forms. By \cite[Lemma 4]{Li75}, the cuspforms with coefficients \eqref{ES} are in fact on the full modular group $\SLZ$.

With a little more work we can even specify the cusp form. We refer to  \cite{ZagierUtrecht} for more details: it holds that
\begin{align}
\label{EStrace}-\frac 12\sum\limits_{s\in\Z} g^{(1)}_\nu(s,n)H(4n-2^2)-\lambda_{2\nu+1}(n)&=\tr(T_n^{(2\nu+2)}(1))\\
\label{Cohtrace}-3\sum\limits_{s\in\Z} g^{(4)}_\nu(s,n)H(n-2^2)-3\lambda_{2\nu+1}(n)&=\tr(T_n^{(2\nu+2)}(4)),
\end{align}
where $T_n^{(k)}(N)$ denotes the $n$th Hecke operator acting on the space $S_k(\Gamma_0(N))$. Note that \eqref{Cohtrace} is only valid for odd $n$. 

The equation \eqref{EStrace} is well-known indeed as the Eichler-Selberg trace formula. Equation \eqref{Cohtrace} is to the author's knowledge first explicitly mentioned (without proof) in \cite{Mert13}.

In order to prove these two trace formulas \eqref{EStrace} and \eqref{Cohtrace} one may use the Rankin-Selberg unfolding trick to see that for any normalized Hecke eigenform $f$ on $\SLZ$ (resp. $\Gamma_0(4)$) we get for $\nu\geq 1$ (cf. \Cref{lem:proppihol})
\[\langle [\widehat{\calH},\vartheta]_\nu,f\rangle=\langle \pihol([\widehat{\calH},\vartheta]_\nu),f\rangle\overset{.}{=}\langle f,f\rangle\]
(for $\SLZ$ one has to apply $U(4)$ to obtain a cusp form of level $1$). The Hecke trace generating function 
\[\mathcal{T}_{2\nu+2}=\sum\limits_{n=1}^\infty \tr(T_n^{(2\nu+2)})q^n\]
is the sum over all normalized Hecke eigenforms, hence we also have $\langle\mathcal{T}_{2\nu+2},f\rangle=\langle f,f\rangle$ since Hecke eigenforms are orthogonal (actually, the summation should be restricted to the $n$ coprime to the level). Therefore $\pihol([\widehat{\calH},\vartheta]_\nu))\overset{.}{=}\mathcal{T}_{2\nu+2}$ which proves both trace formulas.
\subsection{Class Number Relations}
In \cite{Betal}, B. Brown et. al. conjectured a number of nice identities involving class numbers and divisor sums, which have been proven recently by K. Bringmann and B. Kane in \cite{BK13}: 

Let for an integer $a$ and an odd prime $p$
\begin{gather}\label{eq:Hap}
H_{a,p}(n):=\sum\limits_{\substack{s\in\Z \\ s\equiv a\pmod{p}}} H(4n-s^2).
\end{gather}
Then for a prime $\ell$ and $p=5$ it holds that (cf. \cite{BK13}, Equation (4.3))
\begin{gather}\label{eq:Ha5}
H_{a,5}(\ell)=\begin{cases} \frac{\ell+1}{2} & \text{if } a\equiv 0\pmod{5}\text{ and }\ell\equiv 1\pmod{5} \\
                            \frac{\ell+1}{3} & \text{if } a\equiv 0\pmod{5}\text{ and }\ell\equiv 2,3\pmod{5} \\
                            \frac{\ell+1}{3} & \text{if } a\equiv \pm 1\pmod{5}\text{ and }\ell\equiv 1,2\pmod{5} \\ 
                            \frac{5\ell+5}{12} & \text{if } a\equiv \pm 1\pmod{5}\text{ and }\ell\equiv 4\pmod{5} \\ 
                            \frac{5\ell-7}{12} & \text{if } a\equiv \pm 2\pmod{5}\text{ and }\ell\equiv 1\pmod{5} \\
                            \frac{\ell+1}{3} & \text{if } a\equiv \pm 2\pmod{5}\text{ and }\ell\equiv 3,4\pmod{5}.   \end{cases}
\end{gather}
They also give a similar result for $p=7$.

Their proof is to consider the generating function of $H_{a,p}$ which is a mixed mock modular form of weight $2$ on some subgroup of $\SLZ$ and to construct an Appell-Lerch sum which has the same non-holomorphic part as the completion of this generating function, and thus to look at identities of Fourier coefficients of holomorphic modular forms. Then it is easy to compare the first few coefficients (\cite{Kil}, Theorem 3.13 gives an explicit bound for the number of coefficients to be checked) to obtain the general result. 

With \Cref{theo:main32} we can consider more general types of sums as in \eqref{eq:Hap} without too much more work. Let therefore
\begin{gather}\label{eq:Hapnu}
H_{a,p}^{(\nu)}(n):=\sum\limits_{\substack{s\in\Z \\ s\equiv a\pmod{p}}} g_\nu^{(1)}(s,n)H(4n-s^2)
\end{gather}
with $g_\nu^{(1)}$ as in \eqref{ES} and $a,p$ as before. This is up to a constant factor the coefficient of $q^n$ in the Fourier expansion of the function $([\calH,\vartheta^{(p,a)}]_\nu\vert U(4))$ with $U$ as in \eqref{eq:U} and
\[\vartheta^{(p,a)}(\tau):=\sum\limits_{\substack{n\in\Z \\ n\equiv a\pmod{p}}} q^{n^2}.\]
From \Cref{theo:main32} we can now deduce that for
\[\Lambda^{(p,a)}_\nu(\tau):=\sum\limits_{\pm}\large[2\sum\limits_{\substack{m^2-n^2>0 \\ m,n\geq 1 \\ m\equiv \pm a\pmod{p}}} (m-n)^{2\nu+1}q^{m^2-n^2} + \sum\limits_{\substack{m\geq 1 \\ m\equiv \pm a\pmod{p}}} m^{2\nu+1}q^{m^2}\big]\]
the function $([\calH,\vartheta^{(p,a)}]_\nu)+\Lambda^{(p,a)})\vert U(4)$ is a holomorphic cusp form of weight $2+2\nu$ on some group $\Gamma\leq\SLZ$ (to be precise, $\Gamma=\Gamma_0(p^2)\cap\Gamma_1(p)$ if $a\neq 0$ and $\Gamma=\Gamma_0(p^2)$ if $a=0$, see Lemma 3.1 of \cite{BK13}) if $\nu>0$ and a quasi-modular form of weight 2 on the same group if $\nu=0$. Note that by sieving out Fourier coefficients one can turn a quasi-modular form into a holomorphic modular form.

With a little bit of elementary number theory we get the following nice representation for $\Lambda^{(p,a)}_\nu\vert U(4)$, much like in our last example.
\begin{proposition}\label{prop:Lambdapa}
Let 
\[D^{(p,a)}_k(\tau):=\sum\limits_{n=1}^\infty \lambda_k^{(p,a)}(n)q^n,\]
where
\[\lambda_k^{(p,a)}(n):=\sum\limits_{\substack{d\mid n\\ d\leq\sqrt{n}\\ d\equiv -a\pmod{p}}} d^k+\sum\limits_{\substack{d\mid n\\ d\lneqq\sqrt{n}\\ d\equiv a\pmod{p}}} d^k.\]
Then it holds that for $\nu\in\N_0$ we have for $a\neq 0$
\begin{align*}
\left((\Lambda_{2\nu+1}^{(p,a)}\vert U(4)\right)=&2^{2\nu+1}\left[\sum\limits_{b\not\equiv\pm a}\left(D_{2\nu+1}^{\left(p,\frac{a-b}{2}\right)}|S_{p,\frac{a^2-b^2}{4}}\right)(\tau)+\left(\left(D_{2\nu+1}^{(p,a)}+D_{2\nu+1}^{(p,-a)}\right)|S_{p,0}\right)(\tau)\right.\\
&\qquad\qquad\left.+p^{2\nu+1}(D_{2\nu+1}^{(1,0)}|V(p))(\tau)\right].
\end{align*}
and
\[\left((\Lambda_{2\nu+1}^{(p,0)}\vert U(4)\right)=2^{2\nu+1}\cdot\left[\sum\limits_{b\not\equiv 0} \left(D_{2\nu+1}^{\left(p,\frac{b}{2}\right)}|S_{p,-\frac{b^2}{4}}\right)+p^{2\nu+1}\left(D_{2\nu+1}^{(1,0)}|V(p^2)\right)\right]\]
otherwise.
\end{proposition}
A proof of this may be found in Proposition V.4.3. of \cite{MertPhD}.

This together with \Cref{theo:main32} is a generalization of Theorem 1.4 of \cite{BK13}.
\begin{corollary}\label{cor:ex}
\begin{enumerate}[(i)]
\item \Cref{prop:Lambdapa} yields for $p=5$ and $a=0$ that
\[(\calH\vartheta^{(5,0)})\vert U(4)+5D_1^{1,0}\vert V(25)+2D_1^{(5,1)}\vert S_{5,4}+2D_1^{(5,2)}\vert S_{5,4}\]
is a quasi modular form of weight $2$ on $\Gamma_0(25)$. By comparing Fourier coefficients one can find that this function equals
\[\frac{1}{2}G_2 +\frac{1}{12}G_2\otimes \chi_5(1-\chi_5)-G_2\vert V(5)+\frac 52 G_2\vert V(25)\]
where $G_2=-\tfrac{1}{24}+\sum\limits_{n=1}^\infty \sigma_1(n)q^n$ denotes the Eisenstein series of weight $2$, $\chi_p$ stands for the non-trivial real-valued character modulo $p$, and  $V$ and $\otimes \chi$ are as in  \eqref{eq:V} and \eqref{eq:tensor} respectively.  
\item For $p=7$ and $a=0$ we find as above that
\begin{align*}
&(\calH\vartheta_{0,7})\vert U(4)+7D_1^{1,0}\vert V(49) +2D_1^{(7,2)}\vert S_{7,3}+2D_1^{(7,4)}\vert S_{7,5} +2D_1^{(7,1)}\vert S_{7,6}\\
 =&\frac{1}{4}G_2 -\frac{1}{24}G_2\otimes \chi_7(1-\chi_7)+\frac 14 g_7
\end{align*}
where $g_7$ represents the cusp form of weight $2$ associated to the elliptic curve over $\Q$ with Weierstra{\ss} equation $y^2=x^3- 2835x - 71442$ (Cremona label 49a1, cf. \cite{lmfdb}).
\end{enumerate}
\end{corollary}
Note that \Cref{cor:ex} contains the assertions from Corollary 4.3 and Corollary 4.5 of \cite{BK13}. Furthermore, the remaining cases of the class number relations conjectured in \cite{Betal} can now be handled easily by essentially comparing Fourier coefficients. 
\section*{Acknowledgements}
The author's research is supported by the DFG Graduiertenkolleg 1269 ``Global Structures in Geometry and Analysis'' at the Universit\"at zu K\"oln.

The author would like to thank Prof. Dr. Kathrin Bringmann and Prof. Dr. Ken Ono for suggesting this project and Prof. Dr. Sander Zwegers, Dr. Ben Kane, Dr. Larry Rolen and Ren\'e Olivetto for several helpful discussions and comments.


\begin{thebibliography}{10}
\bibitem{AA14}
S.~Ahlgren and N.~Andersen.
\newblock Euler-like recurrences for smallest parts functions.
\newblock {\em preprint, arXiv:1402.5366}, 2014.

\bibitem{AAR}
G.~E. Andrews, R.~Askey, and R.~Roy, \emph{Special functions}, Encyclopedia of
  Mathematics and its Applications, vol.~71, Cambridge University Press, 2000.

\bibitem{ARZ}
G.~E. Andrews, R.~C. Rhoades, and S.~P. Zwegers.
\newblock Modularity of the concave composition generating function.
\newblock {\em Algebra and Number Theory}, to appear.

\bibitem{BR95}
B.~C. Berndt and R.~A. Rankin, \emph{Ramanujan: {L}etters and {C}ommentary},
  History of Mathematics, vol.~9, Americ. Math. Soc., 1995.

\bibitem{BF13}
K.~Bringmann and A.~Folsom, \emph{Almost harmonic {M}aass forms and {K}ac
  {W}akimoto characters}, J. reine und angew. Math. (accepted for publication).

\bibitem{BK13}
K.~Bringmann and B.~Kane, \emph{Sums of class numbers and mixed mock modular
  forms}, preprint, arXiv:1305.0112 [math.NT] (2013).

\bibitem{BL09}
K.~Bringmann and J.~Lovejoy, \emph{Overpartitions and class numbers of binary
  quadratic forms}, Proc. Natl. Acad. Sci. USA \textbf{106} (2009), 5513--5516.

\bibitem{BO06}
K.~Bringmann and K.~Ono, \emph{{T}he $f(q)$ mock theta function conjecture and
  partition ranks}, Invent. Math. \textbf{165} (2006), 243--266.

\bibitem{BO10}
K.~Bringmann and K.~Ono, \emph{Dyson's rank and {M}aass forms}, Ann. of Math. \textbf{171}
  (2010), 419--449.

\bibitem{Betal}
B.~Brown, N.~J. Calkin, T.~B. Flowers, K.~James, E.~Smith, and A.~Stout,
  \emph{Elliptic {C}urves, {M}odular {F}orms, and {S}ums of {H}urwitz {C}lass
  {N}umbers}, J. of Number Theory \textbf{128} (2008), 1847--1863.

\bibitem{BF04}
J.~H. Bruinier and J.~Funke, \emph{{O}n two geometric theta lifts}, Duke Math.
  J. \textbf{1} (2004), no.~125, 45--90.

\bibitem{Coh75}
H.~Cohen, \emph{Sums {I}nvolving the {V}alues at {N}egative {I}ntegers of
  {$L$}-{F}unctions of {Q}uadratic {C}haracters}, Math. Ann. \textbf{217}
  (1975), 271--285.

\bibitem{Zag12}
A.~Dabolkhar, S.~Murthy, and D.~Zagier, \emph{Quantum {B}lack holes, {W}all
  {C}rossing, and {M}ock {M}odular {F}orms}, preprint, arXiv:1208.4074 [hep-th]
  (2012).

\bibitem{Eichler55}
M.~Eichler, \emph{On the {C}lass {N}umber of {I}maginary {Q}uadratic {F}ields
  and the {S}ums of {D}ivisors of {N}atural {N}umbers}, J. Ind. Math. Soc.
  \textbf{15} (1955), 153--180.

\bibitem{lmfdb}
J.~Cremona et~al., \emph{{LMFDB}, the database of ${L}$-functions, modular
  forms, and related objects}, \url{http://www.lmfdb.org/}.

\bibitem{FL05}
W.~Fischer and I.~Lieb, \emph{Funktionentheorie}, Vieweg-Verlag, 9. edition,
  2005.

\bibitem{FB06}
E.~Freitag and R.~Busam, \emph{Funktionentheorie 1}, 4. ed., Springer-Verlag,
  2006.

\bibitem{GrZ86}
B.~H. Gross and D.~B. Zagier, \emph{Heegner points and derivatives of
  {$L$}-series}, Invent. Math. \textbf{84} (1986), 225--320.

\bibitem{HZ76}
F.~Hirzebruch and D.~Zagier, \emph{{I}ntersection {N}umbers of {C}urves on
  {H}ilbert {M}odular {S}urfaces and {M}odular {F}orms of {N}ebentypus}, Inv.
  Math. \textbf{36} (1976), 57--113.

\bibitem{Hur84}
A.~Hurwitz, \emph{{\"U}ber {R}elationen zwischen {K}lassenzahlen bin\"arer
  quadratischer {F}ormen von negativer {D}eterminante}, Berichte der
  k\"oniglich s\"achsischen Gesellschaft der Wissenschaften zu Leipzig,
  mathematisch-physikalische Klasse \textbf{36} (1884), 193--197,
  (Mathematische Werke Bd. 2, pp. 1--4).

\bibitem{Hur85}
A.~Hurwitz, \emph{{\"U}ber {R}elationen zwischen {K}lassenzahlen bin\"arer
  quadratischer {F}ormen von negativer {D}eterminante}, Math. Ann. \textbf{25}
  (1885), 157--196, (Mathematische Werke Bd. 2, pp. 8--50).

\bibitem{IRR13}
\"O. Imamo\u{g}lu, M.~Raum, and O.~Richter, \emph{Holomorphic projections and
  {R}amanujan's mock theta functions}, preprint, arXiv:1306.3919 [math.NT]
  (2013).

\bibitem{KZ95}
M.~Kaneko and D.~B. Zagier, \emph{A generalized {J}acobi theta function and
  quasi modular forms}, The Moduli Space of Curves (Robbert Dijkgraaf, Carel
  Faber, and Gerard B.~M. van~der Geer, eds.), Progr. Math., vol. 129,
  Birkh\"{a}user, 1995, pp.~165--172.

\bibitem{Kil}
L.~J.~K. Kilford, \emph{Modular {F}orms - {A} {C}lassical and {C}omputational
  {I}ntroduction}, Imperial College Press, 2008.

\bibitem{Kro60}
L.~Kronecker, \emph{{\"U}ber die {A}nzahl der verschiedenen {K}lassen
  quadratischer {F}ormen von negativer {D}eterminante}, Journal Reine Angew.
  Math. \textbf{57} (1860), 248--255, (Werke, Bd. IV, pp. 185--195).

\bibitem{Li75}
W.~Li, \emph{Newforms and {F}unctional {E}quations}, Math. Ann. \textbf{212}
  (1975), 285--315.

\bibitem{Mert13}
M.~H. Mertens, \emph{Mock {M}odular {F}orms and {C}lass {N}umber {R}elations},
  preprint, arXiv:1305.5122 [math.NT] (2013).

\bibitem{MertPhD}
M.~H. Mertens, \emph{Mock {M}odular {F}orms and {C}lass {N}umbers of {Q}uadratic
  {F}orms}, Ph.D. thesis, Universit\"at zu K\"oln, 2014.

\bibitem{Ono}
K.~Ono, \emph{Unearthing the visions of a master: harmonic {M}aass forms and
  number theory}, Current Developments in Mathematics \textbf{2008} (2009),
  347--454.

\bibitem{SS77}
J-P. Serre and H.~M. Stark, \emph{Modular {F}orms of {W}eight 1/2}, Modular
  Functions of One Variable VI (Jean-Pierre Serre and Don~Bernard Zagier,
  eds.), Lecture Notes in Mathematics, vol. 627, Springer Berlin Heidelberg,
  1977, pp.~27--67.

\bibitem{Shi}
G.~Shimura, \emph{Modular {F}orms of {H}alf {I}ntegral {W}eight}, Ann. of Math.
  (2) \textbf{97} (1973), no.~3, 440--481.

\bibitem{St80}
J.~Sturm, \emph{Projections of ${C}^\infty$ automorphic forms}, Bull. Amer.
  Math. Soc. (N.S.) \textbf{2} (1980), no.~3, 435--439.

\bibitem{ZagierUtrecht}
D.~Zagier, \emph{Modular {F}orms of {O}ne {V}ariable}, Lecture Notes,
  Universiteit Utrecht,
  \url{http://people.mpim-bonn.mpg.de/zagier/files/tex/UtrechtLectures/UtBook.%
pdf}, 1991.

\bibitem{ZagNotes}
D.~Zagier, \emph{Mock {M}odular {F}orms - {T}heory and {E}xamples}, Talk at the
  international conference: Mock theta functions and applications in
  combinatorics, algebraic geometry, and mathematical physics, MPIM Bonn, May
  25-29, 2009.

\bibitem{ZwegersDiss}
S.~Zwegers, \emph{Mock {T}heta {F}unctions}, Ph.D. thesis, Universiteit
  Utrecht, 2002.

\bibitem{Zwegers10}
S.~Zwegers, \emph{Multivariable {A}ppell {F}unctions}, preprint (2010).

\end{thebibliography}
\end{document}